\documentclass[12pt]{article}
\usepackage[english]{babel}
\usepackage{graphicx}
\usepackage{amsmath}
\usepackage{amsfonts}
\usepackage{amssymb}
\usepackage{amsthm}
\usepackage{hyperref}
\usepackage[T1]{fontenc}
\usepackage{color}
\usepackage{mathtools}
\usepackage{comment}
\usepackage{multicol}

\input xy
\xyoption{all}

\newtheorem{thm}{Theorem}[section]
\newtheorem{cor}[thm]{Corollary}
\newtheorem{lem}[thm]{Lemma}
\newtheorem{pro}[thm]{Proposition}

\theoremstyle{definition}
\newtheorem{defi}[thm]{Definition}

\newtheoremstyle{remarque}{}{}{}{}{\it}{.}{\newline}{}
\theoremstyle{remarque}
\newtheorem*{rem}{Remark}

\newcommand{\asd}[5]{%
\setbox1=\hbox{\ensuremath{^{#1}}}%
\setbox2=\hbox{\ensuremath{_{#2}}}%
\setbox5=\hbox{\ensuremath{#5}}%
\hspace{\ifnum\wd1>\wd2\wd1\else\wd2\fi}%
\ensuremath{\copy5^{\hspace{-\wd1}\hspace{-\wd5}#1\hspace{\wd5}#3}%
_{\hspace{-\wd2}\hspace{-\wd5}#2\hspace{\wd5}#4}%
}}

\usepackage[OT2,T1]{fontenc}
\DeclareSymbolFont{cyrletters}{OT2}{wncyr}{m}{n}
\DeclareMathSymbol{\Sha}{\mathalpha}{cyrletters}{"58}
\DeclareMathSymbol{\Brusse}{\mathalpha}{cyrletters}{"42}

\newcommand{\n}{\mathbb{N}}
\newcommand{\z}{\mathbb{Z}}
\newcommand{\q}{\mathbb{Q}}
\renewcommand{\r}{\mathbb{R}}
\newcommand{\C}{\mathbb{C}}
\newcommand{\f}{\mathbb{F}}

\renewcommand{\ker}{\mathrm{Ker }}

\renewcommand{\hom}{\mathrm{Hom}}
\newcommand{\res}{\mathrm{Res}}
\renewcommand{\inf}{\mathrm{Inf}}
\newcommand{\aut}{\mathrm{Aut}\,}

\newcommand{\gm}{\mathbb{G}_{\mathrm{m}}}
\newcommand{\br}{\mathrm{Br}\,}
\newcommand{\un}{\mathrm{un}}
\newcommand{\brnr}{\mathrm{Br}_{\un}}
\newcommand{\bro}{\mathrm{Br}_0}
\newcommand{\brone}{\mathrm{Br}_1}
\newcommand{\al}{\mathrm{al}}
\newcommand{\bral}{\mathrm{Br}_{\al}}

\newcommand{\brnral}{\mathrm{Br}_{\un,\al}}
\newcommand{\gal}{\mathrm{Gal}}
\newcommand{\sln}{\mathrm{SL}_n}
\newcommand{\pic}{\mathrm{Pic}\,}

\newcommand{\sha}[3]{\Sha^{#1}_{\mathrm{#2},\mathrm{#3}}}
\newcommand{\subgps}[3]{\asd{}{\mathrm{#1}}{}{\mathrm{#2}}{(#3)}}
\newcommand{\ab}{\mathrm{ab}}
\newcommand{\cyc}{\mathrm{cyc}}
\newcommand{\scyc}{\mathrm{scyc}}
\newcommand{\bic}{\mathrm{bic}}

\renewcommand{\int}{\mathrm{int}}

\usepackage{marvosym}

\title{The unramified Brauer group of homogeneous spaces with finite stabilizer}
\author{Giancarlo Lucchini Arteche\\[5mm]
{\it\small Departamento de Matem\'aticas, Facultad de Ciencias, Universidad de Chile}\\
{\it\small Las Palmeras 3425, \~Nu\~noa, Santiago, Chile}\\
{\small luco@uchile.cl}
}

\date{}

\begin{document}

\maketitle

\begin{abstract}
We give formulas for calculating the unramified Brauer group of a homogeneous space $X$ of a semisimple simply connected group $G$ with finite geometric stabilizer $\bar F$ over a wide family of fields of characteristic 0. When $k$ is a number field, we use these formulas in order to study the Brauer-Manin obstruction to the Hasse principle and weak approximation. We prove in particular that the Brauer-Manin pairing is constant on $X(k_v)$ for every $v$ outside from an explicit finite set of non archimedean places of $k$.
{\bf MSC classes:} 14F22, 14M17, 14G20, 14G25.
\end{abstract}

\section{Introduction}
The Brauer group $\br X$ of a (smooth, geometrically integral) variety $X$ over a field $k$ is defined as the second \'etale cohomology group $H^2(X,\gm)$. A variant of this group is the \emph{unramified Brauer group} $\brnr X\subset\br X$. The latter has the advantage of being a stably birational invariant (and as such, it admits a definition that only depends on the fields $k(X)$ and $k$, cf.~\cite{ColliotSantaBarbara}). Moreover, it corresponds to the Brauer group of a smooth compactification of $X$ and as such it can be used for the definition of the Brauer-Manin obstruction to the Hasse principle and weak approximation. Since these arithmetic properties and the corresponding obstruction involve both global and local fields, it is important to develop formulas for $\brnr X$ over arbitrary fields.\\

In this paper we are interested in the unramified Brauer group of homogeneous spaces of (connected) linear groups. In this context, when $X$ is a homogeneous space of a linear group $G$ with connected or abelian stabilizer (in general, of ``ssumult'' type), Borovoi, Demarche and Harari have given formulas for the ``algebraic part'' of $\brnr X$ in \cite{BDH}. The same part in the case of arbitrary stabilizers was covered by the author in \cite{GLABrnral2}, but only for a semisimple simply connected group $G$. It is worth noting that, for connected and abelian stabilizers, the algebraic part corresponds to the whole group $\brnr X$ (i.e.~there is no ``transcendental part''). This is not true anymore for general non connected stabilizers, in particular for finite stabilizers, as it was proven over $k=\C$ by Bogomolov in \cite{BogomolovBrnr1}. Passing then to the whole group $\brnr X$ is a completely different task and it is only recently that Colliot-Th\'el\`ene managed to do this in \cite{ColliotBrnr} for homogeneous spaces of $G=\sln$ with finite \emph{constant} stabilizer (in particular $X(k)\neq\emptyset$) by generalizing Bogomolov's method. In this paper, we use the same methods in order to generalize this result in three directions:
\begin{itemize}\setlength{\itemsep}{0em}
\item we replace $\sln$ with a semisimple simply connected group (although this is a simple application of a remark by Colliot-Th\'el\`ene himself, cf.~the appendix of \cite{GLABrnral2});
\item we do not assume the finite stabilizer to be constant;
\item we do not even assume that there are $k$-points in $X$.
\end{itemize}
The price we have to pay for such a generalization is a small restriction on the base field: namely we have to avoid a certain family of fields which, for lack of a better term, we have named \emph{essentially real fields}. These include all real-closed fields, but not much more, cf.~section \ref{section ess real}. In particular, all global fields and non-archimedean local fields, for which we want to find arithmetic applications, are considered in our formulas. Moreover, the general intermediate results in this paper are enough to get the expected application for the field of real numbers.\\

The structure of the article is as follows. In section \ref{section notations} we fix notations, give the definition of essentially real fields and we study some subgroups of a finite gerb which are necessary to state the main formulas. In section \ref{section FX} we interpret the group $\br X$ in terms of group cohomology using a natural gerb that is associated to any homogeneous space $X$ (which is simply the fundamental group of $X$). We also reinterpret the Brauer-Manin pairing in terms of restrictions in group cohomology. In section \ref{section abelian} we study the particular case of finite abelian stabilizers, for which we prove the triviality of $\brnr X$ under some extra hypotheses. This is the heart of Bogomolov and Colliot-Th\'el\`ene's method. We finally get to the main formulas in section \ref{section main thm} and give an arithmetic application in section \ref{section app}, which consists in giving an explicit set of bad places away from which the Brauer-Manin obstruction does not detect anything.

\paragraph*{Acknowledgements} This work was partially supported by Fondecyt Grant 11170016 and PAI Grant 79170034.

\section{Notations and preliminaries}\label{section notations}
All throughout the text, $k$ denotes a field of characteristic 0. We denote $\bar k$ an algebraic closure of $k$ and $\Gamma_k$ the absolute Galois group $\gal(\bar k/k)$. For $L/k$ any other Galois extension, we denote the corresponding Galois group by $\Gamma_{L/k}$. For $X$ a $k$-variety, we denote by $\bar X$ the $\bar k$-variety $X\times_k\bar k$. In particular, we will be denoting every $\bar k$-variety by a letter with a bar. We will denote by $G$ a semisimple simply connected group and by $F$ (resp.~$\bar F$) a finite $k$-group (resp.~$\bar k$-group). For such a finite group, \emph{we will make the following abuse of notation}: we will always identify the finite algebraic group with the abstract group of its $\bar k$-points (eventually with an action of $\Gamma_k$), thus avoiding everywhere the notation $F(\bar k)$ or $\bar F(\bar k)$. We think that the context will be clear enough for the reader to see whether we mean the group or the scheme.

We will denote by $X$ or $Y$ a $k$-variety, typically a homogeneous space. For such a variety, the Brauer group is defined as $\br X:=H^2_{\text{\'et}}(X,\gm)$. This group comes with a filtration
\[\bro X\subset \brone X\subset \br X,\]
where $\bro X$ corresponds to the image of the morphism $\br k\to\br X$ induced by the structure morphism and $\brone X$ corresponds to the kernel of the morphism $\br X\to\br \bar X$ induced by base change. The quotient $\bral X:=\brone X/\bro X$ is called the algebraic Brauer group. The unramified Brauer group $\brnr X$ is a subgroup of $\br X$ that corresponds to the Brauer group of a smooth compactification of $X$ (which always exists due to Hironaka's theorem). For a definition that is only dependent on the function field of $X$, see for instance \cite{ColliotSantaBarbara}. One can similarly define the unramified algebraic group $\brnral X$.

By a \emph{procyclic} group, we mean a commutative profinite group which is topologically generated by a single element. If $\Gamma$ is procyclic, then one can always canonically write $\Gamma\cong\prod_p\Gamma_p$, where $p$ ranges through the set of all prime numbers and $\Gamma_p$ is a procyclic pro-$p$-group, hence a quotient of $(\z_p,+)$. By a \emph{strictly procyclic} group, we mean a procyclic group $\Gamma$ such that each one of its $p$-Sylow subgroups $\Gamma_p$ is either trivial or infinite (and hence isomorphic to $\z_p$).

\subsection{Essentially real fields}\label{section ess real}
Recall that, by the Artin-Schreier Theorem, an absolute Galois group is finite if and only if it is isomorphic to $\z/2\z$. In particular, a finite closed subgroup of an arbitrary Galois group must be isomorphic to $\z/2\z$.

\begin{defi}
We will say that a field $k$ is \emph{essentially real} if the 2-Sylow subgroups of $\Gamma_k$ are finite and \emph{non trivial} (so a fortiori of order 2).
\end{defi}

Examples of essentially real fields are real closed fields (including $\r$), but also fields like $k=\bigcup_{n\in\n}\r((t^{\frac{1}{2^n}}))$. These fields happen to be the only ones for which our techniques for calculating the unramified Brauer group of a homogeneous space do not work. Note however that every such field has a unique quadratic extension for which our formulas work, so a simple restriction-corestriction argument tells us that we can calculate the prime-to-2 part of the unramified Brauer group for these fields as well. Moreover, aside from the real numbers, there seem to be no other interesting fields if one has arithmetic applications in mind. And we can use some of the results here below that are independent of the field to get the expected applications in the real case, cf.~section \ref{section app}.

\subsection{Subgroups of finite gerbs}
Let $k$ be a field of characteristic 0, $\bar F$ be a finite $\bar k$-group. A $(\bar F,k)$-gerb (or simply a finite gerb) in this context is just a group extension
\[1\to \bar F\xrightarrow{\iota} E \xrightarrow{\pi} \Gamma_k\to 1,\]
such that $\pi$ is continuous. Note that in particular $E$ is a profinite group. Following Springer, one can naturally associate a finite gerb to every homogeneous space with finite stabilizer (cf.~\cite{SpringerH2}). We will do the same thing here below by considering the \'etale fundamental group. Then, in order to calculate the unramified Bruaer group of the homogeneous space, we will consider certain families of subgroups of a given gerb $E$, defined as follows. Recall that a \emph{strictly procyclic} group is a procyclic group such that each one of its $p$-Sylow subgroups is either trivial or infinite.

\begin{defi}\label{defi shas}
Let $E$ be a $(\bar F,k)$-gerb as above and let $A$ be a (discrete) $E$-module. For $\mathrm{x}\in\{\ab,\bic,\cyc\}$ and $\mathrm{y}\in\{\scyc,0\}$, we denote by $\subgps{\mathrm{x}}{\mathrm{y}}{E}$ the set of closed subgroups $D$ of $E$ such that
\begin{itemize}\setlength{\itemsep}{0em}
\item $\iota^{-1}(D)=D\cap \bar F$ is abelian ($\mathrm{x}=\ab$), resp.~bicyclic ($\mathrm{x}=\bic$), resp.~cyclic ($\mathrm{x}=\cyc$);
\item $\pi(D)$ is strictly procyclic ($\mathrm{y}=\scyc$), resp.~trivial ($\mathrm{y}=0$).
\end{itemize}
We also set, for $i\geq 1$:
\begin{align*}
\sha{i}{\mathrm{x}}{\mathrm{y}}(E,A) &:=\ker\left(H^i(E,A)\to \prod_{D\in _\mathrm{x}(E)_\mathrm{y}}H^i(D,A)\right).
\end{align*}
\end{defi}

For example, $\sha{2}{\bic}{\scyc}(E,A)$ is the subgroup of $H^2(E,A)$ of elements that are trivialized by restriction to every closed subgroup $D$ of $E$ such that $\pi(D)$ is strictly procyclic and $D\cap \bar F$ is bicyclic; while $\sha{2}{\ab}{0}(E,A)$ is the subgroup of elements that are trivialized by restriction to every closed abelian subgroup $D$ of $\bar F$ (seen as a subgroup of $E$).

\section{Brauer group and Brauer pairing in group cohomology}\label{section FX}
Let $X$ be a homogeneous space of a semisimple simply connected group $G$ with finite geometric stabilizer $\bar F$. We claim that there is a canonical $(\bar F,k)$-gerb $\bar F^X$ associated to $X$, i.e.~a group extension
\begin{equation}\label{extension FX}
1\to \bar F\to \bar F^X\to \Gamma_k\to 1,
\end{equation}
In fact, such a group can be obtained by taking the \'etale fundamental group of $X$. Consider the Galois cover $\bar X\to X$ and fix a geometric point $\bar x$ of $\bar X$. Then we get an isomorphism $\bar X\cong \bar F\backslash \bar G$, where $\bar F$ is the $\bar k$-group stabilizing $\bar x$. In particular, $\bar G\to\bar X$ is an $\bar F$-torsor and hence a Galois cover. Since $\bar G$ is simply connected, we know by the theory of \'etale covers (see \cite[Exp. V]{SGA1}) that:
\begin{itemize}\setlength{\itemsep}{0em}
\item the group $\pi_1(X,\bar x)$ corresponds to the automorphism group of the cover $\bar G\to X$;
\item the group $\bar F$ can be seen as a subgroup of $\pi_1(X,\bar x)$, corresponding to those automorphisms fixing $\bar X$;
\item since $\bar X\to X$ is Galois, the subgroup $\bar F$ is normal in $\pi_1(X,\bar x)$;
\item the quotient of $\pi_1(X,\bar x)$ by $\bar F$ is isomorphic to $\Gamma_k$, the group of automorphisms of $\bar X\to X$.
\end{itemize}
For such a homogeneous space, we define then $\bar F^X$ to be the group $\pi_1(X,\bar x)$, which fits in the group extension \eqref{extension FX} and hence corresponds to a finite gerb. It is well-known that the isomorphism class of such a group does not depend on the choice of $\bar x$ and it is hence canonically associated to $X$.\\

Note that if $X(k)\neq\emptyset$, then we may consider the \'etale cover $G\to X$ defined by sending $g\in G$ to $xg\in X$ for some $x\in X(k)$. This is in fact an $F$-torsor, where $F$ is the stabilizer of $x$ and hence a $k$-group. Once again, by the theory of \'etale covers, $G\to X$ corresponds to a subgroup of $F^X$ which is easily seen to correspond to a splitting of sequence \eqref{extension FX}. As such, it defines an action of $\Gamma_k$ on $F$ by conjugation in $F^X$ which turns out to be the natural action of $\Gamma$ on the geometric points of $F$ as a $k$-group.

In other words, if $X(k)\neq\emptyset$, the natural extension \eqref{extension FX} associated to $X$ is the semidirect product of $\Gamma_k$ and $F$ for the natural action of $\Gamma_k$ on the stabilizer $F$ of a $k$-point.

\begin{rem}
The class of the extension we have just presented is the \emph{Springer class} associated to the homogeneous space $X$ in the nonabelian 2-cohomology set $H^2(\bar k/k,\bar F\,\mathrm{rel}\,G)$, as defined by Springer in \cite{SpringerH2}. Elements in this set are \emph{gerbs} (which justifies our use of the word), as introduced by Giraud in \cite{Giraud}.

Note also that the use of either the \'etale fundamental group or the Springer class can be avoided in the case where $X(k)\neq\emptyset$, where one can just put $F^X:=F\rtimes\Gamma_k$ by definition. The need for these tools appears in the case where we do not have a rational point, case in which (to our knowledge) no formula for calculating the unramified Brauer group of a homogeneous space with finite stabilizer had been given until now, with the exception of \cite{Demarche} for the algebraic part $\brnral X$ in the particular case where $k$ is a number field and $G=\sln$.
\end{rem}

The proposition below is also a simple consequence of the theory of \'etale covers as explained in \cite[Exp. V]{SGA1}.

\begin{pro}\label{prop subgps}
Let $k$, $G$ and $X$ be as above. Let $E$ be a closed subgroup of $\bar F^X$ and denote $\bar H=E\cap \bar F$. Denote by $L$ the subfield of $\bar k$ fixed by the image of $E$ in $\Gamma_k$. Then the evident $\bar k$-morphism $\bar H\backslash\bar G\to\bar X=\bar F\backslash\bar G$ descends to a $G_L$-equivariant $L$-morphism $Y\to X_L$, where $Y$ is a right homogeneous space of $G_L$ with geometric stabilizer $\bar H$. Moreover, one has $E\cong \bar H^Y$.\qed
\end{pro}

We will use the gerb $\bar F^X$ to study the unramified Brauer group $\brnr X$ as follows.

\begin{pro}\label{prop Brnr et H2}
Let $k$, $X$ and $G$ be as above and let $\bar F^X$ be the canonical gerb associated to $X$. Then
\[\br X=H^2(\bar F^X,\bar k^*),\]
where $\bar F^X$ acts on $\bar k^*$ via its quotient $\Gamma_k$ in the obvious way. Moreover, for any algebraic extension $L/k$ and any homogeneous space $Y$ of $G_L$ with finite stabilizer $\bar H$ equipped with a $G_L$-equivariant map $Y\to X_L$, one has a commutative diagram
\[\xymatrix{
\brnr X \ar[r] \ar[d]^\res & H^2(\bar F^X,\bar k^*) \ar[d]^\res \\
\brnr Y \ar[r] & H^2(\bar H^Y,\bar k^*).
}\]
\end{pro}

\begin{proof}
By construction, we know that the arrow $\bar G\to X$ is an $\bar F^X$-cover. We have the right then to consider the Hochschild-Serre spectral sequence in \'etale cohomology:
\[H^p(\bar F^X,H^q(\bar G,\gm))\Rightarrow H^{p+q}(X,\gm).\]
Since $\pic(\bar G)=0$ and $\bar k[X]^*=\bar k^*$, we deduce from this spectral sequence the following exact sequence
\[0\to H^2(\bar F^X,\bar k^*)\to \br X \to \br \bar G,\]
where we note that $\bar F^X$ acts on $\bar k^*$ via its quotient $\Gamma_k$ in the obvious way. Now, by \cite[\S 0]{SGilleBrGss}, we know that $\br\bar G$ is trivial for semisimple simply connected groups. This gives the equality $\br X=H^2(\bar F^X,\bar k^*)$.

Finally, the unramified Brauer group is functorial for dominant morphisms  (cf.~\cite[Lem.~5.5]{ColliotSansucChili}) and the spectral sequence is compatible with restrictions. This, along with Proposition \ref{prop subgps}, proves the commutativity of the diagram.
\end{proof}

Let us study now how the classic Brauer-Manin pairing
\[\br X\times X(k)\to \br k,\]
behaves with respect to this new interpretation of $\br X$. Consider a point $x\in X(k)$. Then, as we remarked before, we get a homomorphic section $s_x:\Gamma_k\to \bar F^X$, so that $\bar F^X$ becomes a semi-direct product $\bar F\rtimes\Gamma_k$ for the action of $\Gamma_k$ on the stabilizer of $x$, which is a $k$-form of $\bar F$ that we note $F_x$. Such a point allows us moreover to see $X$ as a quotient $G/F_x$, hence we can consider the coboundary morphism $\delta_x:X(k)\to H^1(k,F_x)$.

Now, it is well-known that the set $H^1(k,F_x)$ classifies, up to conjugation, the continuous sections $\Gamma_k\to \bar F^X$. Then any point $y\in X(k)$ gives us as the same time a new homomorphic section $s_y$ of $\bar F^X$ and a class $\delta_x(y)\in H^1(k,F_x)$. The following result is then an easy exercise:

\begin{lem}
Let $y\in X(k)$. Then the class in $H^1(k,F_x)$ corresponding to the homomorphic section $s_y$ is $\delta_x(y)$.\qed
\end{lem}

Thus for $y\in X(k)$, we obtain the following restriction morphism
\[H^2(\bar F^X,\bar k^*)\xrightarrow{s_y^*} H^2(\Gamma_k,\bar k^*)=\br k,\]
which only depends on the class $\delta_x(y)\in H^1(k,F_x)$: indeed, any other point having the same image by $\delta_x$ corresponds to a section that is conjugate to $s_y$ and conjugation induces the identity on $H^2(\bar F^X,\bar k^*)$. We have thus found compatible pairings:
\begin{equation}\label{eq acc GLA}
\xymatrix{
H^2(\bar F^X,\bar k^*) \times \quad X(k)\quad \ar[r] \ar@<7.5ex>[d]^{\delta_x} \ar@<-5ex>@{=}[d] & \br k \ar@{=}[d] \\
H^2(\bar F^X,\bar k^*)\times H^1(k,F_x)   \ar[r] & \br k,
}
\end{equation}
which are easily seen to be functorial in $k$.

\begin{pro}\label{prop acc GLA et BM}
The pairing on the upper line of \eqref{eq acc GLA} is the Brauer-Manin pairing.
\end{pro}

\begin{proof}
We may assume that $X(k)\neq\emptyset$ (otherwise there is nothing to prove). In particular, we may assume that $X=G/F$ for some finite $k$-group $F$ (once again, this amounts to fixing a point in $X(k)$) and hence $\bar F^X=F^X$ is a semi-direct product as above.

Consider now, for $x\in X(k)$, the cartesian commutative diagram
\[\xymatrix@R=1em@C=1em{
\bar P \ar[r] \ar[d] & \bar G \ar[d] \\
P \ar[r] \ar[d] & G \ar[d] \\
\mathrm{Spec}\, k \ar[r]^>>>x & X,
}\]
where $\bar P$ is the trivial $F$-torsor over $\bar k$, i.e.~copies of $\mathrm{Spec}\,\bar k$ permutated by $F$. The diagram induces a morphism of Hochschild-Serre spectral sequences
\[\xymatrix@C=1.5em@R=1em{
H^p(F^X,H^q(\bar G,\gm)) \ar@{=>}[r] \ar[d] & H^{p+q}(X,\gm) \ar[d]^{x^*} \\
H^p(F^X,H^q(\bar P,\gm)) \ar@{=>}[r] & H^{p+q}(k,\gm),
}\]
which gives us in particular the following commutative square
\[\xymatrix@R=1em{
H^2(F^X,H^0(\bar G,\gm)) \ar[d] \ar[r] & H^2(X,\gm) \ar[d]^{x^*} \\
H^2(F^X,H^0(\bar P,\gm)) \ar[r] & H^2(k,\gm).
}\]
The Brauer-Manin pairing for the group $H^2(F^X,\bar k^*)$ with respect to the point $x$ corresponds to passing through the upper part of the diagram, whereas the pairing in \eqref{eq acc GLA} corresponds to passing through the lower part. Indeed, this part of the diagram can be restated as
\[H^2(F^X,\bar k^*)\to H^2(F^X,(\bar k^*)^{\oplus n})\to \br k,\]
where $n$ denotes the order of $F$ and where the left arrow is defined by the diagonal inclusion. A direct computation shows that the $F^X$-module in the middle term is the induced module $\mathrm{Ind}^{s_x(\Gamma_k)}_{F^X}(\bar k^*)$ and hence the middle term is just $H^2(s_x(\Gamma_k),\bar k^*)$ by Shapiro's Lemma. This proves that the composition of both arrows corresponds to the restriction with respect to the subgroup $s_x(\Gamma_k)$ of $F^X$.
\end{proof}

\section{The case of abelian stabilizers}\label{section abelian}
The formulas we present below are inspired by Colliot-Th\'el\`ene's formula in the constant case (cf.~\cite[thm.~5.5]{ColliotBrnr}), where he emulates Bogomolov's work in \cite{BogomolovBrnr1}. In particular, the formula uses the triviality of the unramified Brauer group of homogeneous spaces with \emph{constant abelian} stabilizer. In order to generalize the formula to the non-constant case, we also need to generalize the result on triviality to non-constant abelian stabilizers.\\

Let $k$ be a field of characteristic 0, $G$ be a semisimple simply connected $k$-group and $X$ a homogeneous space of $G$ with finite abelian stabilizer $\bar A$. We get then as before the finite gerb
\[1\to \bar A\to \bar A^X\to \Gamma_k\to 1,\]
which gives us a natural action of $\Gamma_k$ on $\bar A$ by conjugation in $\bar A^X$ (which is well defined since $\bar A$ is \emph{abelian}). We get then a natural $k$-form $A$ of $\bar A$ associated to $X$.

\begin{pro}\label{prop abelian}
Let $k$ be a field of characteristic 0, $G$ be a semisimple simply connected $k$-group and $X$ a homogeneous space of $G$ with finite abelian stabilizer. Let $A$ denote the natural $k$-form of the stabilizer, $L/k$ be an extension splitting $A$ and $n$ be the exponent of $A$. Assume that for every $p^r$ dividing $n$ with $p$ a prime number, the extension $L(\mu_{p^r})/k$ is cyclic. Then $\brnr X=\bro X$.
\end{pro}

Recall (cf.~Section \ref{section notations}) that $\bro X$ was defined as the image of $\br k\to\br X$, so this proposition says that the unramified Brauer group consists in \emph{constant} classes under such hypotheses.

\begin{rem}
Note that if $L=k$ the last hypothesis is automatically satisfied for odd $p$, so that it is only a condition on the ciclicity of the extension $k(\mu_{2^r})/k$. As such, this is a generalization of Colliot-Th\'el\`ene's condition $Cyc(G,k)$ for a constant $k$-group $G$, which is precisely the particular case of $L=k$, see \cite{ColliotBrnr}.
\end{rem}

\begin{proof}
Since $\brnr X$ is a torsion group, it will suffice to prove the proposition for $\brnr X\{p\}$ for a given prime number $p$. Consider then the maximal subextension $k\subset L_p\subset L$ such that $[L_p:k]$ is prime to $p$ (note that $L/k$ is necessarily cyclic). By the classic restriction-corestriction argument, we may assume that $k=L_p$ and hence that $\Gamma_{L/k}$ is a finite cyclic $p$-group.

Consider now the $p$-Sylow subgroup $S$ of $A$. Since the group $H^2(k,A)$ from which $\bar A^X$ comes from splits into the product of its Sylow subgroups, we know that there exists an extension
\[1 \to S \to E \to \Gamma_{k} \to 1,\]
that corresponds, on one side, to a closed subgroup of $\bar A^X$ of index prime to $p$, and on the other side, by Proposition \ref{prop subgps}, to the group $\bar S^Y$ for a homogeneous space $Y$ of $G$ with geometric stabilizer $\bar S$ lying above $X$. Proposition \ref{prop Brnr et H2} tells us then that the restriction of $\brnr X\{p\}$ to $H^2(\bar S^Y,\bar k^*)$ is contained in $\brnr Y\{p\}$ and once again the classic restriction-corestriction argument tells us that we may reduce to the case where $X=Y$, i.e.~we may assume that $A$ is an abelian $p$-group split by a cyclic $p$-primary extension.\\

Since $A$ is abelian, we get that $\brnr \bar X=0$ by \cite[Prop.~26]{GLABrnral2} and \cite[Prop.~3.2]{ColliotBrnr}, hence any element $\alpha\in\brnr X$ is algebraic. By \cite[Thm.~8.1]{BDH}, we know then that the image $\bar\alpha$ of $\alpha$ in $\brnr X/\bro X$ is in $\Sha^1_\cyc(k,\hat A)\subset H^1(k,\hat A)$, where $\hat A=\hom(A,\mu_{p^r})$, $p^r$ is the exponent of $A$ and $\Sha^1_\cyc$ denotes the elements that are trivialized by restriction to every procyclic subgroup of $\Gamma_k$. Now, the inflation-restriction sequence for $L(\mu_{p^r})/k$ reads
\[1\to H^1(L(\mu_{p^r})/k,\hat A)\to H^1(k,\hat A)\to H^1(L(\mu_{p^r}),\hat A),\]
and $\bar\alpha\in H^1(k,\hat A)$ is clearly sent to $\Sha^1_\cyc(L(\mu_{p^r}),\hat A)$. Since $\hat A$ is constant over $L(\mu_{p^r})$, the term on the right corresponds to homomorphisms $\Gamma_{L(\mu_{p^r})}\to \hat A$, hence $\Sha^1_\cyc(L(\mu_{p^r}),\hat A)$ is trivial and $\bar\alpha$ comes from $H^1(L(\mu_{p^r})/k,\hat A)$. But since $L(\mu_{p^r})/k$ is cyclic by condition 2, there exists a procyclic subgroup $C$ of $\Gamma_k$ surjecting onto $\Gamma_{L(\mu_{p^r})/k}$, so that we have the following commutative diagram
\[\xymatrix@R=1.2em{
H^1(L(\mu_{p^r})/k,\hat A) \ar@{^{(}->}[r]^>>>>>>\inf \ar@{^{(}->}[dr]_\inf & H^1(k,\hat A) \ar[d]^\res \\
& H^1(C,\hat A).
}\]
Since the restriction of $\bar\alpha$ to $C$ is trivial, we deduce that $\bar\alpha$ is trivial, hence $\alpha\in\bro X$, which concludes the proof.
\end{proof}

\section{General formulas for $\brnr X$}\label{section main thm}
Let $k$ be a field of characteristic $0$, $G$ a semisimple simply connected $k$-group and $X$ a homogeneous space of $G$ with finite geometric stabilizer $\bar F$. Let $\bar F^X$ be the finite gerb canonically associated to $X$, cf.~section \ref{section FX}. Recall that Proposition \ref{prop Brnr et H2} gives us an isomorphism
\[\br X\xrightarrow{\sim} H^2(\bar F^X,\bar k^*).\]
Moreover, if we denote by $\mu$ the subgroup of $\bar k^*$ consisting of all the roots of unity, then we have an exact sequence
\[1\to\mu\to\bar k^*\to\bar k^*/\mu\to 1,\]
in which the group $\bar k^*/\mu$ is a uniquely divisible $\bar F^X$-module, hence \emph{cohomologically trivial}. This means that $H^2(\bar F^X,\mu)=H^2(\bar F^X,\bar k^*)$
and hence we can replace $\bar k^*$ by $\mu$ on the statement here below.

\begin{thm}\label{main thm}
Let $k$ be a field that is \emph{not essentially real} (i.e.~the 2-Sylows of $\Gamma_k$ are either infinite or trivial). Let $G,X,\bar F$ be as above. Then the following subgroups of $H^2(\bar F^X,\bar k^*)$ coincide and correspond to $\brnr X$:
\begin{multicols}{2}
\begin{itemize}\setlength{\itemsep}{0em}
\item $\sha{2}{ab}{scyc}(\bar F^X,\bar k^*)$;
\item $\sha{2}{cyc}{scyc}(\bar F^X,\bar k^*)\cap\sha{2}{ab}{0}(\bar F^X,\bar k^*)$;
\item $\sha{2}{bic}{scyc}(\bar F^X,\bar k^*)$;
\item $\sha{2}{cyc}{scyc}(\bar F^X,\bar k^*)\cap\sha{2}{bic}{0}(\bar F^X,\bar k^*)$.
\end{itemize}
\end{multicols}
\end{thm}

\begin{rem}
In \cite[Thm.~5.5]{ColliotBrnr}, Colliot-Th\'el\`ene proves that in the case where the stabilizer of a $k$-point is a constant $k$-group $F$ (i.e.~when $\bar F^X=\bar F\times\Gamma_k$) the normalized unramified Brauer group of $X$ is isomorphic to some unspecified subgroup of $\Sha^2_\ab(\bar F,k^*)$. Using the inflation arrow $H^2(\bar F,k)\to H^2(\bar F^X,\bar k^*)$, which is easily seen to be injective in this particular case, we see that $\Sha^2_\ab(\bar F,k^*)$ falls into $\sha{2}{ab}{0}(\bar F^X,\bar k^*)$. Then Theorem \ref{main thm} gives a first description of this subgroup: it is the intersection with $\sha{2}{cyc}{scyc}(\bar F^X,\bar k^*)$.
\end{rem}

\begin{proof}[Proof of Theorem \ref{main thm}]
Note that the following inclusions are evident
\[\xymatrix@R=1em@C=1em{
\sha{2}{ab}{scyc}(\bar F^X,\bar k^*) \ar@{^{(}->}[r] \ar@{^{(}->}[d] & \sha{2}{bic}{scyc}(\bar F^X,\bar k^*) \ar@{^{(}->}[d] \\
\sha{2}{cyc}{scyc}(\bar F^X,\bar k^*)\cap\sha{2}{ab}{0}(\bar F^X,\bar k^*) \ar@{^{(}->}[r] & \sha{2}{cyc}{scyc}(\bar F^X,\bar k^*)\cap\sha{2}{bic}{0}(\bar F^X,\bar k^*).
}\]
Let us then prove first the inclusion
\[\brnr X\subset \sha{2}{ab}{scyc}(\bar F^X,\bar k^*).\]
Let $\alpha\in\brnr X$. Consider a closed subgroup $E$ of $\bar F^X$ in $\subgps{ab}{scyc}{\bar F^X}$. By Proposition \ref{prop subgps}, we know that $E=\bar H^Y$ for some homogeneous space $Y$ of $G_L$ whose geometric stabilizer $\bar H=E\cap\bar F$ is abelian, where $L$ is an extension of $k$ such that $\Gamma_L$ is strictly procyclic. We have then that the restriction $\alpha_E$ of $\alpha$ to $H^2(E,\bar k^*)$ falls into $\brnr Y$ by Proposition \ref{prop Brnr et H2}. Now, Proposition \ref{prop abelian} tells us that $\brnr Y=\br_0 Y$. But since $\Gamma_L$ is strictly procyclic, the field $L$ is of cohomological dimension $\leq 1$ and hence $\br L=0$.\footnote{This is a point where our proof fails for essentially real fields, since in general one should consider all procyclic subgroups, for which we may then get $\br L\simeq\z/2\z$.} Thus $\alpha_E$ is trivial for all such $E$ and $\alpha\in\sha{2}{ab}{scyc}(\bar F^X,\bar k^*)$ as claimed.\\

Let us prove now the inclusion
\[\sha{2}{cyc}{scyc}(\bar F^X,\bar k^*)\cap\sha{2}{bic}{0}(\bar F^X,\bar k^*)\subset\brnr X.\]
Let $\alpha\in H^2(\bar F^X,\bar k^*)\subset\br X$ and assume that $\alpha\not\in\brnr X$. We need to show that there exists a subgroup $E$ of $\bar F^X$ in either $\subgps{bic}{0}{\bar F^X}$ or $\subgps{cyc}{scyc}{\bar F^X}$ such that the restriction $\alpha_E\in H^2(E,\bar k^*)$ is non trivial. Recall that the Brauer group of $X$ is embedded in $\br k(X)$, where $k(X)$ denotes the function field of $X$. Let then $A\subset k(X)$ be a discrete valuation ring of rank one with fraction field $k(X)$, residue field $\kappa_A\supset k$ and such that $\partial_A(\alpha)\neq 0$, where $\partial_A: \br k(X)\to H^1(\kappa_A,\q/\z)$ is the residue map (cf.~\cite[Def.~5.2]{ColliotSansucChili}).

We will follow Bogomolov's argument as it is presented in \cite[Thm.~6.1]{ColliotSansucChili}. In particular, the facts that follow come from \cite[I.7]{SerreCorpsLocaux}. Let $\tilde A$ be the integral closure of $A$ in $\bar k(\bar G)$.\footnote{Note that we are dealing here with an \emph{infinite} Galois extension $\bar k(\bar G)/k(X)$, but since the infinite part falls in the residue field extension, the reader will easily see through this ``problem''.} It is a semi-local Dedekind ring. We can choose then a prime ideal and consider the localization $B$ of $\tilde A$ in it, which is also a discrete valuation ring with residue field $\kappa_B\supset \bar k$. Let $D\subset \bar F^X$ be the associated decomposition subgroup and $I\subset D$ the corresponding inertia group, that is, the kernel of the surjective morphism $D\to\gal(\kappa_B/\kappa_A)$. Since $k\subset \kappa_A$ and $\bar k\subset\kappa_B$, it is easy to see that $I$ is contained in $\bar F$. The group $I$ is thus finite and, by \cite[IV.2]{SerreCorpsLocaux}, we deduce that it is cyclic and central in $D\cap \bar F$ since the field $\bar k(\bar G)^{\bar F}=\bar k(\bar X)$ contains all the roots of unity. On the other hand, one easily sees that $D$ is of finite index in $\bar F^X$, hence its 2-Sylow subgroups are either infinite or contained in $\bar F$ by our hypothesis on $\Gamma_k$.

Denote by $\alpha_I$ the restriction of $\alpha$ to $H^2(I,\bar k^*)$ and similarly for $D$. If $\alpha_I\neq 0$, then we are done since $I$ is cyclic and hence clearly belongs to $\subgps{bic}{0}{\bar F^X}$. Assume then that $\alpha_I=0$. Consider the tower of discrete valuation rings $A\subset B^D\subset B^I\subset B$ with respective fraction fields $k(X)\subset \bar k(\bar G)^D\subset \bar k(\bar G)^I\subset \bar k(\bar G)$ and respective residue fields $\kappa_A=\kappa_A\subset\kappa_B=\kappa_B$. By \cite[Prop.~3.3.1]{ColliotSantaBarbara}, we have the following commutative diagram
\[\xymatrix{
\br k(X) \ar[r] \ar[d]_{\partial_A} &\br \bar k(\bar G)^{D} \ar[r] \ar[d]_{\partial_{B^D}} & \br \bar k(\bar G)^{I}  \ar[d]_{\partial_{B^I}} \\
H^1(\kappa_A,\q/\z) \ar@{=}[r] & H^1(\kappa_A,\q/\z) \ar[r]^{\res} & H^1(\kappa_B,\q/\z),
}\]
with $\partial_A(\alpha)\neq 0$ and $\partial_{B^I}(\alpha_I)=0$. Now, recall that $H^1(\kappa,\q/\z)=\hom(\Gamma_{\kappa},\q/\z)$, hence there exists an element $\bar f\in \gal(\kappa_B/\kappa_A)\cong D/I$ such that $\partial_A(\alpha)(\bar f)\neq 0$. Take a preimage $f\in D$ and consider the closed procyclic subgroup $C:=\langle f\rangle\subset D\subset \bar F^X$. We claim that we may choose $f$ so that either $C\subset \bar F$ or $C\cap \bar F=\{1\}$.

Assuming the claim, put $E:=\langle I,C\rangle\subset D$ and insert $\br \bar k(G)^E$ in the diagram above. We verify then that $\partial_{B^{E}}(\alpha_{E})\neq 0$, which implies that $\alpha_{E}\neq 0$. Now, since $I$ is normal in $D$ and central in $D\cap \bar F$, one easily deduces that either $E\in\subgps{bic}{0}{\bar F^X}$ or $E\in\subgps{cyc}{scyc}{\bar F^X}$ depending on the choice above. Given the trivial inclusions stated at the beginning, this proves the theorem.\\

We give now a proof of the claim in order to finish. The group $C$ splits into a direct product of pro-$p$-groups $C^p$ such that:
\begin{itemize}\setlength{\itemsep}{0em}
\item if $C^p$ is infinite, then $C^p\cap \bar F=\{1\}$;
\item if $C^p$ is finite, then $C^p\subset \bar F$, except maybe if $p=2$.
\end{itemize}
Indeed, if $C^p$ is infinite, then $C^p\cong\z_p$ and the intersection $C^p\cap F$ corresponds to a finite subgroup of $\z_p$, hence it is trivial. On the other hand, if $C^p$ is finite, then its image in the quotient $\Gamma_k$ is finite and hence trivial if $p\neq 2$ by the Artin-Schreier theorem. Since $\partial_A(\alpha)$ is non trivial on $C$, it must be non trivial on one of the $C^p$, so that up to changing $f$ by a generator of such a $C^p$ we are done, unless the only such prime is $p=2$ and $C^2$ is finite. In that case, either the 2-Sylow of $D$ is contained in $\bar F$ and hence so is $C^2$ and we are done, or else the 2-Sylow of $D$ is infinite. In this last case, we can modify $f$ by a $2$-primary element of infinite order in the kernel of $\partial_A(\alpha)$, so that we get an infinite $C^2$, which allows us to conclude.
\end{proof}

We now give a lemma that will be useful for the applications in section \ref{section app}. Recall that, given $x\in X(k)$, we may \emph{normalize} the Brauer group $\br X$ by considering the subgroup of elements that are trivialized by evaluation at $x$, that is, we may define
\[\br^xX:=\ker[\br X\xrightarrow{x^*} \br k],\]
and define the normalized unramified Brauer group $\brnr^x X$ as the intersection of $\br^xX$ with $\brnr X$.

\begin{lem}\label{lem H2 mu}
Let $k$ be a field of characteristic 0 that is \emph{not} essentially real, $G$ a semisimple simply connected $k$-group, $F\subset G$ a finite $k$-subgroup of order $n$ and $X$ the $k$-homogeneous space $G/F$. Let $x\in X(k)$ be the point corresponding to the subgroup $F$, so that the natural section $\Gamma_k\to F^X=F\rtimes\Gamma_k$ is $s_x$. Finally, let $\mu\{n\}$ denote the group generated by the $p$-primary roots of unity for every $p$ dividing $n$ and let $L/k$ be a Galois extension splitting $F$ and containing $\mu\{n\}$.

Then the subgroup $s_x(\Gamma_L)$ is normal in $F^X$ and every $\alpha$ in the normalized unramified Brauer group $\brnr^x X\subset H^2(F^X,\bar k^*)$ can be lifted to the group $H^2(F\rtimes\Gamma_{L/k},\mu\{n\})$ via the composite map
\[H^2(F\rtimes\Gamma_{L/k},\mu\{n\})\xrightarrow{\inf}H^2(F^X,\mu\{n\})\to H^2(F^X,\bar k^*),\]
where $F\rtimes\Gamma_{L/k}$ is the quotient of $F^X$ by $s_x(\Gamma_L)$.
\end{lem}

\begin{proof}
The fact that $s_x(\Gamma_L)$ is normal comes from the fact that $\Gamma_L$ is normal in $\Gamma_k$ and acts trivially on $F$ by definition. Then the subgroup $F\rtimes\Gamma_L$ of $F^X$ is actually a direct product and the normality of $s_x(\Gamma_L)$ follows. It is easy to see then that the quotient is $F\rtimes\Gamma_{L/k}$.\\

Concerning the second assertion of the lemma, note that, by definition, $\alpha\in\brnr^x X\subset H^2(F^X,\bar k^*)$ is trivial when restricted to the subgroup $s_x(\Gamma_k)$ of $F^X$. Since this subgroup is of index $n$, the classic restriction-corestriction argument tells us that $\alpha$ is an $n$-torsion element. Consider then the short exact sequence
\[1\to\mu\{n\}\to\bar k^*\to\bar k^*/\mu\{n\}\to 1,\]
and note that the quotient is a \emph{uniquely} $p$-divisible group for every $p$ dividing $n$. This means that multiplication by $p$ is an automorphism of $H^i(F^X,\bar k^*/\mu\{n\})$ for every $i\geq 1$. In particular, these groups have no $p$-primary torsion, which tells us that $\alpha$ can be \emph{uniquely} lifted to an $n$-torsion element of $H^2(F^X,\mu\{n\})$ and so it goes for \emph{any} restriction of $\alpha$ to \emph{any} closed subgroup of $F^X$. We will thus make an abuse of notation and still call $\alpha$ the element in $H^2(F^X,\mu\{n\})$.\\

Consider now the Hochschild-Serre spectral sequence
\[E_2^{p,q}=H^p(F\rtimes\Gamma_{L/k},H^q(L,\mu\{n\}))\Rightarrow H^{p+q}(F^X,\mu\{n\})=E^{p+q},\]
associated to the short exact sequence
\[1\to\Gamma_L\xrightarrow{s_x} F^X\to F\rtimes\Gamma_{L/k}\to 1.\]
Since $L$ contains $\mu\{n\}$, $\Gamma_L$ acts trivially on $\mu\{n\}$ and hence the map $E_2^{2,0}\to E^2$ corresponds to the inflation map of the statement of the theorem. In order to prove the lemma, we have to prove then that $\alpha\in E^2$ maps to 0 in $E_2^{0,2}=H^2(L,\mu\{n\})^{F\rtimes\Gamma_{L/k}}$ and then to 0 in $E_2^{1,1}=H^1(F\rtimes\Gamma_{L/k},H^1(L,\mu\{n\}))$.

Now, note that the map $E^2\to E_2^{0,2}$ is nothing but the restriction map
\[H^2(F^X,\mu\{n\})\xrightarrow{s_x^*} H^2(\Gamma_L,\mu\{n\}),\]
associated to the inclusion $s_x:\Gamma_L\to F^X$. Now, it is by hypothesis that $\alpha$ is trivialized by restriction to $s_x(\Gamma_k)\supset s_x(\Gamma_L)$, so $\alpha$ is indeed trivial in $E_2^{0,2}$. For the second arrow, note that the section $s_x$ is well defined on the quotient $F\rtimes\Gamma_{L/k}$ and the Hochschild-Serre spectral sequence is compatible with restrictions, so that we have a commutative diagram
\[\xymatrix{
\ker(H^2(F^X,\mu\{n\})\to H^2(L,\mu\{n\})) \ar[d]_{s_x^*} \ar[r] & H^1(F\rtimes\Gamma_{L/k},H^1(L,\mu\{n\})) \ar[d]^{s_x^*} \\
\ker(H^2(k,\mu\{n\})\to H^2(L,\mu\{n\})) \ar[r] & H^1(L/k,H^1(L,\mu\{n\})),\\
}\]
which tells us that the image of $\alpha$ in $H^1(L/k,H^1(L,\mu\{n\}))$ is trivial because the restriction of $\alpha$ to $H^2(k,\mu\{n\})$ is. Now, the split exact sequence
\[\xymatrix{
1 \ar[r] & F \ar[r] & F\rtimes\Gamma_{L/k} \ar[r] & \Gamma_{L/k} \ar[r] \ar@/_1pc/[l]_{s_x} & 1,
}\]
induces an inflation-restriction exact sequence with a retraction
\[\xymatrix{
0 \ar[r] & H^1(L/k,H^1(L,\mu\{n\})) \ar[r] & H^1(F\rtimes\Gamma_{L/k},H^1(L,\mu\{n\})) \ar[r] \ar@/_1.5pc/[l]_{s_x^*} & H^1(F,H^1(L,\mu\{n\})),
}\]
which tells us that $H^1(L/k,H^1(L,\mu\{n\}))$ is a direct factor of $H^1(F\rtimes\Gamma_{L/k},H^1(L,\mu\{n\}))$ and since the restriction of $\alpha$ to this factor is trivial, all we are left to prove is that its restriction to $H^1(F,H^1(L,\mu\{n\}))$ is trivial. Consider then the sum of restriction maps
\[H^1(F,H^1(L,\mu\{n\}))\to\bigoplus_{A,K}H^1(A,H^1(K,\mu\{n\})),\]
where $A$ ranges over all abelian subgroups of $F$ and $K$ ranges over all the extensions of $L$ that correspond to a (strictly) procyclic subgroup of $\Gamma_L$. Since all actions are trivial, all $H^1$ groups correspond to $\hom$ groups and hence it is an easy exercise to see that this map is \emph{injective}. Finally, once again, since the spectral sequence is compatible with restrictions, for each such pair $(A,K)$ we have a commutative diagram
\[\xymatrix{
\ker(H^2(F^X,\mu\{n\})\to H^2(L,\mu\{n\})) \ar[d]_{\res} \ar[r] & H^1(F\rtimes\Gamma_{L/k},H^1(L,\mu\{n\})) \ar[d]^{\res} \\
\ker(H^2(A\times\Gamma_K,\mu\{n\})\to H^2(K,\mu\{n\})) \ar[r] & H^1(A,H^1(K,\mu\{n\})),\\
}\]
and since $\alpha\in\brnr^x X$, we know by Proposition \ref{prop Brnr et H2} and Theorem \ref{main thm} that the restriction of $\alpha$ to $H^2(A\times\Gamma_K,\mu\{n\})$ is trivial (recall that pre-images in $H^2(\ast,\mu\{n\})$ are \emph{unique} for the $n$-torsion). This proves the triviality of the image of $\alpha$ in the direct sum $\bigoplus_{A,K}H^1(A,H^1(K,\mu\{n\}))$ and hence its triviality in $H^1(F,H^1(L,\mu\{n\}))$, which concludes the proof of the lemma.
\end{proof}

\section{An arithmetic application: the set of bad places for the Brauer-Manin obstruction}\label{section app}
Let $k$ be a number field, $\Omega_k$ the set of its places, $G$ a semisimple simply connected $k$-group and $X$ a homogeneous space of $G$ with finite geometric stabilizer. When $X$ has a $k$-point, we gave in \cite[\S2.2]{Neftin} a definition of ``bad places'' for $X$ with respect to the Brauer-Manin obstruction to weak approximation that we recall here. Denote by $F$ the stabilizer of a $k$-point in $X$ and by $L/k$ the Galois extension obtained via the kernel of the natural morphism $\Gamma_k\to \aut(F)$ given by the action of $\Gamma_k$ on the points of $F$. Then $v\in\Omega_k$ is said to be a \emph{bad place} if it is a non archimedean place that either ramifies in $L/k$ or divides the order of $F$. Otherwise, we say that $v$ is a \emph{good place}.\\

Note that by this definition all archimedean places are assumed to be good places. This is due to the fact that the Brauer-Manin obstruction doesn't interact with archi\-medean places. In fact, we prove that this is true for \emph{all} good places via the following two theorems:

\begin{thm}\label{thm ev reals}
Let $G$ a semisimple simply connected $\r$-group, $F\subset G$ a finite $\r$-subgroup of order $n$ and $X$ the $\r$-homogeneous space $G/F$. Then the Brauer pairing
\[\brnr X\times X(\r)\to \br \r,\]
is constant on $X(\r)$.
\end{thm}

\begin{thm}\label{thm ev non arch}
Let $k$ be a non archimedean local field of residue characteristic $p$, $G$ a semisimple simply connected $k$-group, $F\subset G$ a finite $k$-subgroup of order $n$ and $X$ the $k$-homogeneous space $G/F$. Assume that $p$ is prime to the order of $F$ and that $F$ is split by an unramified extension of $k$. Then the Brauer pairing
\[\brnr X\times X(k)\to \br k,\]
is constant on $X(k)$.
\end{thm}

With these results, we can relate Colliot-Th\'el\`ene's conjecture on the Brauer-Manin obstruction (cf.~\cite[Introduction]{CT}) with a conjecture given by Demarche, Neftin and the author in \cite[\S1.2]{Neftin} on ``tame approximation'', as it was suggested in \cite[\S2.5]{Neftin}. Indeed, Theorems \ref{thm ev reals} and \ref{thm ev non arch} tell us that the Brauer set $(\prod_{\Omega_k}X(k_v))^{\brnr}$ surjects onto the product $\prod_{S}X(k_v)$ for every finite set $S$ of good places. Assuming the first conjecture, we get then the density of $X(k)$ in $\prod_{S}X(k_v)$. Recalling \cite[Prop.~2.4]{Neftin}, we have thus proved that:

\begin{cor}
Assume that the Brauer-Manin obstruction is the only obstruction to weak approximation for homogeneous spaces with finite stabilizers. Then the \emph{Tame approximation problem} has a positive solution for every finite $k$-group $F$, i.e.~the natural restriction map
\[H^1(k,F)\to\prod_{v\in S}H^1(k_v,F),\]
is surjective for \emph{every} finite set $S$ of good places.\qed
\end{cor}

\begin{proof}[Proof of Theorem \ref{thm ev reals}]
Let $x\in X(\r)$ be the point corresponding to the subgroup $F$. Then we may consider the normalized subgroup $\brnr^x X\subset\brnr X$ of the elements that are trivial when evaluated in $x$. It is clear then that $\brnr X\simeq \bro X\times \brnr^x X$ and hence it will suffice to prove that the pairing is \emph{trivial} on the subgroup $\brnr^x X$.

Consider the group $F^X=F\rtimes\Gamma_{\r}$. This group comes with a splitting $s_x:\Gamma_{\r}\to F^X$ associated to $x$. Consider then an element $\alpha\in\brnr^x X\subset H^2(F^X,\C^*)$ and take a point $y\in X(\r)$. Let $s_y:\Gamma_{\r}\to F^X$ denote the corresponding splitting. We know then by Proposition \ref{prop acc GLA et BM} that the image of $(\alpha,y)$ by the Brauer-Manin pairing is the restriction of $\alpha$ to $H^2(\Gamma_{\r},\C^*)=\br \r$ via $s_y$.

Let $\sigma$ be the non trivial element in $\Gamma_{\r}$. Then, since it is well-known that sections $\Gamma_{\r}\to F^X$ are classified by the cocycle set $Z^1(\r,F)$, we know that $s_y(\sigma)$ is of the form $fs_x(\sigma)$ with $f\in F$ such that $\asd{\sigma}{}{}{}{f}=f^{-1}$. This tells us that there is a cyclic $\sigma$-stable subgroup $C=\langle f\rangle$ of $F$ such that the section $s_y$ factors through $C\rtimes \Gamma_{\r}$. In other words, we have the folllowing commutative diagram of finite groups
\[\xymatrix@R=0em{
& F^X \\
C\rtimes\Gamma_{\r} \ar@{^{(}->}[ur] \\
& \Gamma_{\r} \ar@{^{(}->}[ul]^{s_y} \ar@{^{(}->}[uu]_{s_y} ,
}\]
which induces the following commutative diagram of cohomology groups
\[\xymatrix@R=0em{
& H^2(F^X,\C^*) \ar@{=}[r] \ar[dd]^{s_y^*} \ar[dl]_{\res} & \br X \ar[dd]^{\mathrm{ev}_y} \\
H^2(C\rtimes\Gamma_{\r},\C^*) \ar[dr]_{s_y^*} & \\
& H^2(\Gamma_{\r},\C^*) \ar@{=}[r] & \br \r.
}
\]
On the other hand, since $C$ is $\sigma$-stable, it corresponds to an $\r$-subgroup of $F$. We can consider then the morphism $Z=G/C\to G/F=X$ and the point $z\in Z$ above $x$ corresponding to the subgroup $C$. Since $\alpha\in\brnr^x X$, we know by Proposition \ref{prop Brnr et H2} that its restriction $\alpha_Z$ to $\br Z=H^2(C\rtimes\Gamma_{\r},\C^*)$ actually falls into $\brnr^z Z$. Now, since $C$ and $\Gamma_{\r}$ are cyclic, Proposition \ref{prop abelian} tells us that $\alpha_Z$ is in $\bro Z\cap\brnr^z Z$ and hence is trivial. We conclude then that its restiction to $s_y(\Gamma_{\r})$ is trivial and hence so is its evaluation at $y$ for \emph{every} $y\in X(\r)$. This proves the theorem.
\end{proof}

\begin{proof}[Proof of Theorem \ref{thm ev non arch}]
Let $x\in X(k)$ be the point corresponding to the subgroup $F$. As before, we may consider the normalized subgroup $\brnr^x X\subset\brnr X$ and prove that the pairing is \emph{trivial} on $\brnr^x X$.

Consider the group $F^X=F\rtimes\Gamma_{k}$. This group comes with a splitting $s_x:\Gamma_{k}\to F^X$ associated to $x$ and Proposition \ref{prop acc GLA et BM} tells us then that the elements in $\brnr^x X$ are elements that are trivialized when restricted to $s_x(\Gamma_{k})$. Since the index of this subgroup is $n$, the classic restriction-corestriction argument tells us then that $\brnr^x X$ is an $n$-torsion group.

Let $\mu\{n\}\subset k^*$ be the group generated by the $q$-primary roots of unity for every $q$ dividing $n$. Then it is easy to see that $H^2(\ast,\mu\{n\})$ corresponds to the subgroup of $H^2(\ast,\bar k^*)$ generated by the elements of $q$-primary torsion. In particular, $H^2(k,\mu\{n\})$ conatins the $n$-torsion of $\br k$ and $H^2(F^X,\mu\{n\})$ contains the $n$-torsion of $\br X$.\\

Consider then an element $\alpha\in\brnr^x X\subset H^2(F^X,\mu\{n\})$ and take a point $y\in X(k)$. Let $s_y:\Gamma_{k}\to F^X$ denote the corresponding splitting. We know then by Proposition \ref{prop acc GLA et BM} that the image of $(\alpha,y)$ by the Brauer-Manin pairing is the restriction of $\alpha$ to $H^2(\Gamma_{k},\mu\{n\})\subset\br k$ via $s_y$. Moreover, if we denote by $W\subset\Gamma_{k}$ the subgroup of wild ramification, then $(s_y)|_{W}=(s_x)|_{W}$. Indeed, $W$ is a pro-$p$-group which acts trivially on $F$ (and hence sections of $F\times W$ are classified by group morphisms $W\to F$) and the order of $F$ is prime to $p$.\\

Let now $L/k$ be the maximal unramified extension of $k$. This extension splits $F$ and contains $\mu\{n\}$, so that Lemma \ref{lem H2 mu} tells us that $s_x(\Gamma_L)$ is normal in $F^X$ and we have a projection $\pi_x^L:F^X\to F\rtimes\Gamma_{L/k}$. Let us look then at the image of the morphism
\[\pi_x^L\circ s_y:\Gamma_k\to F\rtimes\Gamma_{L/k}.\]
The image of $W\subset\Gamma_{k}$ is trivial since $s_y(W)=s_x(W)\subset s_x(\Gamma_L)=\ker(\pi_x^L)$. And since it is well-known that there is an isomorphism of profinite groups
\[\Gamma_{k}/W\simeq\langle\sigma,\tau\,|\,\sigma\tau\sigma^{-1}=\tau^q\rangle,\]
where $q$ is the order of the residue field of $k$, we see that we only have to care about the images of $\sigma$ and $\tau$, which represent respectively a lift of the Frobenius (which generates $\Gamma_{L/k}$) and a generator of the tame inertia subgroup (cf.~\cite[7.5.3]{NSW}).

Since $L/k$ is unramified and $s_y$ is a section, the image of $\tau$ by $\pi_x^L\circ s_y$ must be in $F$ and generate a $s_y(\sigma)$-stable cyclic subgroup $T$ of $F$.

Consider then a lift of $\sigma$ to $\Gamma_k$ and the procyclic subgroup $C\subset\Gamma_k$ generated by it. Then $C$ is an extension of $\langle\sigma\rangle\simeq\hat\z$ and some pro-$p$ procyclic group $P$. We have the following commutative diagram of profinite groups
\[\xymatrix@R=.5em{
F\rtimes\Gamma_{L/k} && F^X \ar@{->>}[ll]_{\pi_x^L} \\
& Ts_y(C) \ar@{^{(}->}[ur] \ar@{->>}[dl]_{\pi_x^L} \\
\pi_x^L(s_y(\Gamma_{k})) \ar@{^{(}->}[uu] && \Gamma_{k} \ar@{->>}[ll]^{\pi_x^L\circ s_y} \ar@{^{(}->}[uu]_{s_y} ,
}\]
where $Ts_y(C)$ is a subgroup of $F^X$ whose image in the quotient $\Gamma_{k}$ is $C$ and whose intersection with $F$ is $T$ since $T$ is stable by the action of $s_y(C)$. This diagram induces the following diagram of 2-cohomology groups
\[\xymatrix@R=.5em{
H^2(F\rtimes\Gamma_{L/k},\mu\{n\}) \ar[rr]^{(\pi_x^L)^*} \ar[dd]_{\res} && H^2(F^X,\mu\{n\}) \ar@{^{(}->}[r] \ar[dd]^{s_y^*} \ar[dl]^{\res} & \br X \ar[dd]^{\mathrm{ev}_y} \\
& H^2(Ts_y(C),\mu\{n\}) & \\
H^2(\pi_x^L(s_y(\Gamma_{k})),\mu\{n\}) \ar[rr]_{s_y^*\circ(\pi_x^L)^*} \ar[ur]^{(\pi_x^L)^*} && H^2(\Gamma_{k},\mu\{n\}) \ar@{^{(}->}[r] & \br k,
}
\]
where the $(\pi_x^L)^*$ arrows are inflation arrows since $\pi_x^L$ is surjective. Using Lemma \ref{lem H2 mu}, we can lift $\alpha\in H^2(F^X,\mu\{n\})$ to an element in $H^2(F\rtimes\Gamma_{L/k},\mu\{n\})$ and we know that its image in $H^2(Ts_y(C),\mu\{n\})$ is trivial by Theorem \ref{main thm}. Assume that the map
\[(\pi_x^L)^*:H^2(\pi_x^L(s_y(\Gamma_{k})),\mu\{n\})\to H^2(Ts_y(C),\mu\{n\}),\]
is injective. Then we get that the image of the lift of $\alpha$ in $H^2(\pi_x^L(s_y(\Gamma_{k})),\mu\{n\})$ is trivial, which tells us that its image in $H^2(\Gamma_{k},\mu\{n\})$ is trivial and hence its evaluation at $y$ is trivial, which concludes the proof of the theorem.

Let us then prove the injectivity of $(\pi_x^L)^*$ in order to conclude. This is an inflation map coming from the surjective morphism $Ts_y(C)\to \pi_x^L(s_y(\Gamma_{k}))$, whose kernel is the pro-$p$ group $P$. An application of the Hochschild-Serre spectral sequence
\[H^p(\pi_x^L(s_y(\Gamma_{k})),H^q(P,\mu\{n\}))\Rightarrow H^{p+q}(Ts_y(C),\mu\{n\}),\]
tells us then that injectivity follows from the triviality of $H^1(P,\mu\{n\})$, which is easily seen since $P$ is a pro-$p$-group and $\mu\{n\}$ is a torsion group whose elements are all of order prime to $p$.
\end{proof}

\begin{rem}
All the results in this paper are most probably true for fields of characteristic $p>0$ as long as one assumes that the finite group $\bar F$ is of order prime to $p$. Note that, accordingly with Theorem \ref{thm ev non arch}, this should imply in particular that, for the function field $k$ of a curve over $\f_q$, \emph{there is no Brauer-Manin obstruction to weak approximation} for homogeneous spaces $X$ of $G$ with $X(k)\neq \emptyset$ and finite stabilizer $F$ such that the natural map $\Gamma_k\to\aut F$ factors through $\Gamma_{\f_q}$ (in particular for constant $F$). Proving (or disproving) that such homogeneous spaces have weak approximation should then be a very interesting result.
\end{rem}


\begin{thebibliography}{ABCD00}
\bibitem[Bog87]{BogomolovBrnr1}
{\bf F.~A.~Bogomolov.} The {B}rauer group of quotient spaces of linear representations.
{\it Izv. Akad. Nauk SSSR Ser. Mat.} 51(3), 485--516, 688, 1987.

\bibitem[BDH13]{BDH}
{\bf M.~Borovoi, C.~Demarche, D.~Harari.} Complexes de groupes de type multiplicatif et groupe de {B}rauer non ramifi\'e des espaces homog\`enes. {\it Ann. Sci. Ec. Norm. Sup.} 46, 651--692, 2013.

\bibitem[CT95]{ColliotSantaBarbara}
{\bf J.-L.~Colliot-Th{\'e}l{\`e}ne.} Birational invariants, purity and the {G}ersten conjecture. {\it {$K$}-theory and algebraic geometry: connections with quadratic forms and division algebras ({S}anta {B}arbara, {CA}, 1992)}, 1--64, Proc. Sympos. Pure Math., 58, {\it Amer. Math. Soc., Providence, RI}, 1995.

\bibitem[CT03]{CT}
{\bf J.-L. Colliot-Th\'el\`ene.} Points rationnels sur les fibrations. {\it Higher Dimensional Varieties and Rational Points}, Bolyai Society Mathematical Series, vol 12, Springer-Verlag, edited by K. J. B\"or\"oczky,  J. Koll\`ar and  T. Szamuely, 171--221, 2003.

\bibitem[CT14]{ColliotBrnr}
{\bf J.-L.~Colliot-Th{\'e}l{\`e}ne.} Groupe de {B}rauer non ramifi\'e de quotients par un groupe fini. {\it {P}roc. {A}mer. {M}ath. {S}oc.} 142(5), 1457-1469, 2014.

\bibitem[CTS07]{ColliotSansucChili}
{\bf J.-L.~Colliot-Th{\'e}l{\`e}ne, J.-J.~Sansuc.} The rationality problem for fields of invariants under linear algebraic groups (with special regards to the {B}rauer group). {\it Algebraic groups and homogeneous spaces}, 113--186, Tata Inst. Fund. Res. Stud. Math., {\it Tata Inst. Fund. Res., Mumbai}, 2007.

\bibitem[Dem10]{Demarche}
{\bf C.~Demarche.} Groupe de {B}rauer non ramifi\'e d'espaces homog\`enes \`a stabilisateurs finis. {\it Math. Ann.} 346(4), 949--968, 2010.

\bibitem[DLAN17]{Neftin}
{\bf C.~Demarche, G.~Lucchini {Arteche}, D.~Neftin.} The Grunwald Problem and approximation properties for homogeneous spaces. {\it Ann. Inst. Fourier (Grenoble)} 67(3), 1009--1033, 2017.

\bibitem[Gil09]{SGilleBrGss}
{\bf S.~Gille.} On the Brauer group of a semisimple algebraic group. {\it Adv.~Math.} 220(3), 913--925, 2009.

\bibitem[Gir71]{Giraud}
{\bf J.~Giraud.} {\it Cohomologie non ab\'elienne}. Die Grundlehren der mathematischen Wissenschaften, No. {\bf 179}. Springer-Verlag, Berlin-New York, 1971.

\bibitem[LA15]{GLABrnral2}
{\bf G.~Lucchini {A}rteche.} Groupe de Brauer non ramifi\'e alg\'ebrique des espaces homog\`enes. {\it Transform. Groups} 20, 463--493, 2015.

\bibitem[NSW08]{NSW}
{\bf J.~Neukirch, A.~Schmidt, K.~Wingberg.} {\it Cohomology of number fields}. Grundlehren der Mathematischen Wissenschaften, No. {\bf 323}.
Springer-Verlag, Berlin, second edition, 2008.

\bibitem[SGA1]{SGA1}
{\it Rev\^etements \'etales et groupe fondamental}. S\'eminaire de G\'eom\'etrie Alg\'ebrique du Bois Marie 1960/61. Dirig\'e par A.~Grothendieck.
Lecture Notes in Mathematics, No. {\bf 224}, Springer, Berlin, 1971; Documents Math\'ematiques, No.~{\bf 3}. Soci\'et\'e Math\'ematique de France, Paris, 2003.

\bibitem[Ser79]{SerreCorpsLocaux}
{\bf J.-P.~Serre.} {\it Local fields}. Graduate Texts in Mathematics, No 67. Springer-Verlag, New York, 1979.

\bibitem[Spr66]{SpringerH2}
{\bf T.~A.~Springer.} Nonabelian {$H^{2}$} in {G}alois cohomology. {\it Algebraic {G}roups and {D}iscontinuous {S}ubgroups ({P}roc. {S}ympos. {P}ure {M}ath., {B}oulder, {C}olo., 1965)}, 164--182, {\it Amer. Math. Soc., Providence, RI}, 1966.

\end{thebibliography}
\end{document}